\documentclass{amsart}
\usepackage{xspace,amssymb}
\usepackage{ifpdf}\ifpdf\usepackage{hyperref}\else\usepackage[hypertex]{hyperref}\fi
\usepackage{tikz}
\usetikzlibrary{cd}
\tikzcdset{scale cd/.style={every label/.append style={scale=#1},
    cells={nodes={scale=#1}}}}
\usepackage{pgfplots}
\usepgfplotslibrary{external}
\pgfplotsset{compat=1.17}
\usepackage{commands} 

\begin{document}

\title{Moduli of representations of one-point extensions}

\author{Arif Dönmez}
\email{arif.doenmez@rub.de, arif.doenmez@iuf-duesseldorf.de}
\address{IUF – Leibniz Research Institute for Environmental Medicine, Düsseldorf, Germany}

\author{Markus Reineke}
\email{markus.reineke@rub.de}
\address{Faculty of Mathematics, Ruhr University Bochum, Bochum, Germany}

\begin{abstract}
We study moduli spaces of (semi-)stable representations of one-point extensions of quivers by rigid representations. This class of moduli spaces unifies Grassmannians of subrepresentations of rigid representations and moduli spaces of representations of generalized Kronecker quivers. With homological methods, we find numerical criteria for non-emptiness and results on basic geometric properties, construct generating semi-invariants, expand the {Gel'fand MacPherson} correspondence, and derive a formula for the {Poincar\'e} polynomial in singular cohomology of these moduli spaces.
\end{abstract}
\maketitle

\section{Introduction}
In this paper we construct and study moduli spaces parametrizing isomorphism classes of representations of so-called one-point extensions of path algebras of quivers. This constitutes a class of algebras of global dimension two, for which many of the favourable properties of moduli spaces of representations of quivers still hold. Namely, we find numerical criteria for non-emptiness and results on basic geometric properties, construct generating semi-invariants, expand the {Gel'fand MacPherson} correspondence, and derive a formula for the {Poincar\'e} polynomial in singular cohomology of these moduli spaces. We explicitly apply the developed theory in several examples.

To do this, we fix a path algebra $A=kQ$ of a finite quiver, which we extend by a representation $T$ of $A$ to the one-point extension algebra $A[T]$.
We construct standard projective resolutions for representations of $A[T]$. One of the most important consequences of this is an explicit description of the space $\Ext^2$ of representations, which allows us to conclude its vanishing on so-called full representations (under the assumption of $T$ being rigid). See Theorems \ref{homologicalProp}, \ref{standardRes}, \ref{Ext2Symmetrie}, \ref{ext2aufvolle} for precise formulations. Moreover, the standard resolutions allow us to calculate the Euler form of $A[T]$ in Theorem \ref{DerivationAllg}, Corollary \ref{Derivation}, \ref{eulerform_folgerung}.
After these preparations, we consider the representation varieties of $A[T]$, interpret the found homological properties in this geometric setting, and rediscover some results by {Schofield} and  {Crawley-Boevey} (with different methods) in Theorem \ref{schofieldCBRefindOne}, Corollary \ref{schofieldCBRefindTwo}, \ref{SchofieldHomFormel}. Moreover, in this way we can determine the  {Zariski} tangent space of the representation variety in each point (Theorem \ref{tangentspace}) and conclude that the open subset of full representations is smooth and irreducible (Theorem \ref{homlogieTrifftGeometrie2}).

We follow the GIT approach of {King} in the construction of moduli spaces. For this, we choose a canonical stability condition, such that the resulting spaces unify quiver Grassmannians of subrepresentations of rigid representations (Theorem \ref{gelfand_mac}) and moduli spaces of representations of generalized {Kronecker} quivers. We find a numerical criterion for semistability in Theorem \ref{kingStabKoecherErw}, which allows us to conclude that semi-stable representations are full representations (Corollary \ref{sstSindGenVoll}). In this way, we can apply the above geometric properties of representation varieties to prove that the resulting moduli spaces are smooth, irreducible and of expected dimension in Theorem \ref{GeoPropModuli}.

After this, we prove a relative version of a recursive numerical criterion (\cite{Re}) for non-emptiness of the semi-stable locus in Theorems \ref{verAllgKin}, \ref{rekursiveFormula}. A set of generators for the ring of semi-invariants, closely following  {Schofield} and  {Van den Bergh} (\cite{S2}) is given in section $6$.

Moreover, we find a form of {Gel'fand MacPherson} correspondence in terms of these moduli spaces in  Theorem \ref{gelfand_mac}. We finish the paper by deriving a recursive formula to determine the  {Poincaré} polynomial of the moduli spaces  (Theorem \ref{poincarepol}). 

\subsection*{Acknowledgements} 
The authors are supported by the DFG SFB / Transregio 191 `Symplektische Strukturen in Geometrie, Algebra und Dynamik'.

\section{One-point extensions and their representations}$ $\\
We fix an algebraically closed field $k$. Let $A=kQ$ be the path algebra of a finite quiver $Q$, and let $T$ be a  finite dimensional $A$-module. We consider the \textit{one-point extension of $A$ by $T$} 
\[
A[T] ~:=~\triAlg{A}{k}{T} 
~:=~
\left\lbrace 
\begin{pmatrix}
a & t \\
0 & \lambda
\end{pmatrix}\;:\; a\in A,\,t\in T,\, \lambda\in k\right\rbrace.
\]
Recall that the multiplication in (the $k$-algebra) $A[T]$ is given by the formal matrix multiplication
\[
\begin{pmatrix}
a & t \\
0 & \lambda
\end{pmatrix} \cdot
\begin{pmatrix}
a' & t' \\
0 & \lambda'
\end{pmatrix}
~:=~
\begin{pmatrix}
aa' & at'+t\lambda' \\
0 & \lambda \lambda'
\end{pmatrix}
\] 
for $a,a'\in A,\, \lambda,\lambda'\in k,\,t,t'\in T$ and componentwise addition.

We define a category $\Rep_A(T\otimes_k ?_2 \to ?_1)$ as follows. Its objects $M$ are tuples $(M_1,M_2)$, where $M_1$ is a left $A$-module and $M_2$ is a $k$-vector space, together with a map of $A$-modules $f_M:T \otimes_k M_2 \to M_1$. A morphism $\varphi: M\to N$ is a tuple $(\varphi_1,\varphi_2)$, where $\varphi_1:M_1\to N_1$ is a map of $A$-modules and $\varphi_2:M_2\to N_2$ a $k$-linear map, such that the diagram 
\[
\begin{tikzcd}
T\otimes_k M_2 \arrow[r, "f_M"] \arrow[d, "\varphi_1"]  & M_1 \arrow[d, "\varphi_2"] \\
T\otimes_k N_2 \arrow[r, "f_N"] & N_1 
\end{tikzcd}                                
\] 
commutes. Composition of morphisms is defined componentwise.
\begin{lmm}
The category $\Rep_A(T\otimes_k ?_2 \to ?_1)$ is equivalent to the category of left $A[T]$-modules.
\end{lmm}
\begin{proof}
See for example \cite{Ri}.
\end{proof}

We can easily describe a quiver $\hat{Q}$ and relations $\Rel$ in this situation such that $A[T]\simeq k\hat{Q}/\langle \Rel \rangle$. By extending $Q$ the following way  
\begin{eqnarray*}
\hat{Q}_0 & = &  Q_0 \sqcup \{\infty \},  \\
\hat{Q}_1  & =  & Q_1 ~ \sqcup ~ \bigg\{ \rho_{l, (i)}:\infty \to i \;:\; i\in Q_0,\,l=1,\ldots, \dim T_i \bigg\}, 
\end{eqnarray*}
we obtain the quiver $\hat{Q}$. We obtain the relations $\Rel$ by using the transformation matrices induced by the fixed representation $T$: For each vertex $i\in Q_0$ we choose a $k$-basis $\rho_{1,(i)}, \ldots, \rho_{\dim T_i,(i)}$ of $T_i$ and write 
$$
T(\alpha)\rho_{l,(i)} = \sum_{s=1}^{\dim T_j}\lambda_{s,l}^{(\alpha)}\rho_{s,(j)}
$$
for all $l=1,\ldots, \dim T_i$ and each arrow $(\alpha:i\to j)\in Q_1$.
The coefficients of these linear combinations induce the relations $\Rel$
\begin{equation} \label{eq:relationGabriel}
\alpha \cdot \rho_{l,(i)} ~=~ \sum_{s=1}^{\dim T_j}\lambda_{s,l}^{(\alpha)}\rho_{s,(j)}. 
\end{equation}

Obviously this construction of relations does not depend (up to isomorphism of $k$-algebras) on the basis we choose in $T_i$ for each vertex $i\in Q_0$.

\begin{bsp}\label{BeispielStart} Let $Q=\left(1\to 2\right)$.\\
	By extending the path algebra of $Q$ with $k^3 \xto{ [1~0~0] }k$ we get
	$$
	\begin{tikzcd}		 
    1	 \arrow{r}{m} &2  \\
    0	 \arrow[bend left=50]{u}{\beta} \arrow{u}{\alpha} \arrow[swap,bend right=50]{u}{\gamma}   & 
	\end{tikzcd} \text{~~with $m \beta =  0,\;m \gamma =  0$.}	
	$$

\end{bsp}
\section{Homological properties}
Since $A$ is the path algebra of a quiver $Q$, we can and will identify $A$ with the tensor algebra  $\mathsf{T}_RX$, where $R$ is the semisimple $k$-algebra generated by the vertices of $Q$ and $X$ is the $R$-$R$-bimodule generated (as a $k$-vector space) by the arrows of $Q$.

\subsection{The standard resolution}
Let $\{e_1,\ldots,e_n\}$ be the complete set of primitive orthogonal idempotents given by the length $0$ paths in $A$. Then
\[
\begin{pmatrix} e_1 & 0 \\ 0 & 0 \end{pmatrix},
\ldots,
\begin{pmatrix} e_n & 0 \\ 0& 0 \end{pmatrix},
\begin{pmatrix} 0 & 0 \\ 0& 1 \end{pmatrix}
\]
is a complete set of primitive orthogonal idempotents of $A[T]$.
Let $\tilde{R}$ be the $k$-subalgebra of $A[T]$ given by
\[
\tilde{R} ~~ := ~~
\bop_{i=1}^nk\begin{pmatrix} e_i & 0 \\ 0 & 0 \end{pmatrix}\;\; \op \;\;
k\begin{pmatrix} 0 & 0 \\ 0& 1 \end{pmatrix}.
\]
The multiplication of $A$ resp. $A[T]$ induces the {Eilenberg} sequence \cite[Proposition 2.7.3]{cohn2002further}
\[
\mathfrak{E}(A): 0\to \Omega(A) \xto{\kappa_A} A\otimes_R A \xto{\mu_A} A \to 0,
\]{
resp.
\[ 
\mathfrak{E}(A[T]): 0\to \Omega(A[T]) \to A[T]\otimes_{\tilde{R}} A[T]\xto{\mu_{A[T]}} A[T] \to 0,
\] }
with $\Omega(A):=\ker{\mu_A}$ resp. $\Omega(A[T]):=\ker{\mu_{A[T]}}$.
 
This sequence of
$A$-bimodules resp. $A[T]$-bimodules splits in the category of right $A$-modules resp. $A[T]$-modules.

\begin{thm}\label{homologicalProp}
Let $\tilde{M}=(M,V,f:T\otimes_k V\to M)$ be an $A[T]$-module. Then there is  a short exact sequence
\[
0\to \Omega(A[T])\otimes_{A[T]}\tilde{M}\to A[T]\otimes_{\tilde{R}}\tilde{M}\to\tilde{M}\to 0,
\] 
where
\begin{eqnarray*}
A[T]\otimes_{\tilde{R}}\tilde{M} & = &
\bigg(T\otimes_k V,\,V,\;T\otimes_k V \xto\id T\otimes_k V\bigg)\;\op\\ & 
&\quad \big(A\otimes_R M,0, {0}\big),   \\& &\\
\Omega(A[T])\otimes_{A[T]}\tilde{M} & = & 
\bigg( T\otimes_k V,\,0,\,{0} \bigg) \op \bigg( \Omega(A)\otimes_A M,\,0,\,{0} \bigg). 
\end{eqnarray*}
\end{thm}

\begin{proof}
We obtain the short exact sequence by tensoring the split short exact sequence of right  $A[T]$-modules  $\mathfrak{E}(A[T])$ over $A[T]$ with $\tilde{M}$.

The first equation follows immediately using the definition of $\tilde{R}$. Namely, for $M_0:=V$ we have  
\[
A[T]\tilde{e}_0   =  \begin{pmatrix} 0 & T\\ 0 & k \end{pmatrix} 
\]
and $A[T]\otimes_{\tilde{R}}\tilde{M} =\bop_{i=0}^n A[T]\tilde{e}_i\otimes_k M_i$.

The second equation follows from the following commutative diagram with exact rows:
\[
\begin{tikzcd}[scale cd=.98, column sep=small]
0 \arrow[r] & 0 \arrow[r]\arrow[d] & T\otimes_k V \arrow[r, "\id"]\arrow[d, "(\id\;\; 0)^\top"] & T\otimes_k V \arrow[d, "f"]  \arrow[r] & 0 \\
0 \arrow[r] & (T\otimes_k V) \op \big(\Omega(A)\otimes_A M\big) \arrow[r, "g"] & (T\otimes_k V) \arrow[r, "(f\;\;\mu_M)" ]\op (A\otimes_R M) & M\arrow[r] & 0,   
\end{tikzcd}                                
\]
where
\[
\mathfrak{E}(A)\otimes_A M: 0\to \Omega(A)\otimes_A M \xto{\kappa_M} A\otimes_R M\xto{\mu_M} M \to 0,
\]
and $g$ is defined as
\[
g\bigg((t\otimes v) + (\omega\otimes m) \bigg)~~:=~~
t\otimes v ~~- 1\otimes f(t\otimes v) + \kappa_M(\omega\otimes m) 
\]
for $t\otimes v\in T\otimes_k V,\;\,\omega\otimes m\in \Omega(A)\otimes_A M$.
\end{proof}

\begin{rmk}To simplify the notation, set $\Omega_A=\Omega(A)$ and $N=\Omega_A\otimes_A M$. Combining Theorem \ref{homologicalProp} with the standard resolution
$$0\rightarrow \Omega_A\otimes_A X\rightarrow A\otimes X\rightarrow X\rightarrow 0$$
of $A$-modules, we obtain the following long exact sequence:
\[
\begin{tikzcd}[column sep=tiny, row sep=small]
 0 \arrow[r] & \Omega_A\otimes_A (T\otimes_k V) \arrow[r] &\big(A\otimes_R (T\otimes_k V)\big)\op N \arrow[dr]\arrow[r]  &  A[T]\otimes_{\tilde{R}}\tilde{M} \arrow[r] & \tilde{M}\arrow[r] & 0. \\
& & & \big(T\otimes_k V\big)\;\op\;N\arrow[u] \arrow[dr]	&                                            &  \\
& &	&   0\arrow[u] &0 & \\
\end{tikzcd}                                
\]
\end{rmk}
Using this construction, we obtain:
\begin{thm}\label{standardRes} For $A[T]$-modules $\tilde{M}=(M,V,f:T\otimes_k V \to M)$ we have the standard projective resolution:
\[ 
\begin{tikzcd}
0 \arrow[d]   & 0 \arrow[d] &  0 \arrow[d]  \\ 
P_2(\tilde{M})\arrow[d] :=  & 0 \arrow[d]\arrow[r] & (A\otimes_R X\otimes_R A)\otimes_A(T\otimes V) \arrow[d, "h"]\\  
P_1(\tilde{M})\arrow[d] :=  & 0 \arrow[d]\arrow[r] & \big(A\otimes_R (T\otimes V)\big)\op \big( (A\otimes_RX\otimes A)\otimes_A M\big)\arrow[d, "g"] \\
P_0(\tilde{M})\arrow[d] :=  & T\otimes V \arrow[r, "(\id~~ 0)^\top"]\arrow[d, "\id"]  & (T\otimes V) \arrow[d, "(f\;\;\,\mu_M)" ]\op (A\otimes_R M) \\
~~~\tilde{M}\arrow[d]= & T\otimes V \arrow[r, "f"] \arrow[d]& M\arrow[d] \\
0 & 0  & 0.
\end{tikzcd}
\]
Here $g$ and $h$ are defined as  
\begin{eqnarray*}
g\left(a\otimes (t\otimes v)+  a\otimes x \otimes  m \right) & = & at\otimes v - a\otimes f(t\otimes v)+ ax\otimes m -a\otimes xm,\\
h\left( a\otimes x \otimes (t\otimes v) \right) & = & \left( ax\otimes (t\otimes v)- a\otimes (xt\otimes v) \right)+a\otimes x \otimes f(t\otimes v),
\end{eqnarray*} $ $\\
for $a\otimes (t\otimes v) \in A\otimes_R(T\otimes_k V), \,a\otimes x \otimes  m\in  (A\otimes_R X\otimes_R A)\otimes_A M$ and
$a\otimes x \otimes (t\otimes v) \in (A\otimes_R X\otimes_R A)\otimes_A(T\otimes_k V)$.\\ 
In particular, we have ${\rm gldim}\; A[T]\leq 2$.
\end{thm}

\subsection{Characterizing \texorpdfstring{$\Ext^2$}{Ext2}}$ $\\
Let $\tilde{M}=\big(M,V,f:T\otimes_kV\to M\big), \tilde{N}=\big(N,W,g:T\otimes_kW\to N\big)$ be $A[T]$-modules.

By using the {Eilenberg} sequence we obtain the following commutative diagram with exact rows and columns
\[
\begin{tikzcd}[row sep=large]
& 0 \arrow[d, dotted] &  0 \arrow[d, dotted] & 0 \arrow[d, dotted] & \\
0 \arrow[r]& \Omega(A)\otimes_A \ker{f}  \arrow[r, "\kappa_{\ker{f}}"] \arrow[d, dotted] & A\otimes_R \ker{f} \arrow[r, "\mu_{\ker{f}}"] \arrow[d, dotted] & \ker{f} \arrow[r] \arrow[d, dotted]& 0, \\
0 \arrow[r]& \Omega(A)\otimes_A (T\otimes_k V)  \arrow[r, "\kappa_{T\otimes_k V}"] \arrow[d, dotted]& A\otimes_R (T\otimes_k V) \arrow[r, "\mu_{ (T\otimes_k V)  }"] \arrow[d, dotted] & T\otimes_k V \arrow[r]\arrow[d, dotted, "f'"]& 0, \\
0 \arrow[r]& \Omega(A)\otimes_A \im{f}  \arrow[r, "\kappa_{\im{f}}"] \arrow[d, dotted] & A\otimes_R \im{f} \arrow[r, "\mu_{\im{f}}"] \arrow[d, dotted] & \im{f} \arrow[r] \arrow[d, dotted] & 0.\\
 & 0 & 0  & 0 & 
\end{tikzcd}                                
\]
Note that, since $R$ is a semisimple $k$-algebra, $\Omega(A)$ and $\Omega_A\otimes_A L$ are projective modules for all $A$-modules $L$.

Applying ${\rm Hom}_A(?,N)$, we obtain the following commutative diagram
\[
\begin{tikzcd}[row sep=large, column sep=small]
  \Hom_R(\im{f},N) \arrow[r]\arrow[d] &  \Hom_A(\Omega_A\otimes_A \im{f}, N) \arrow[d,"(\Omega_A\otimes f')^*"]\arrow[r] & \Ext_A(\im{f},N) \arrow[d]\arrow[r] & 0   \\
  \Hom_R(T\otimes_k V, N) \arrow[d]\arrow[r,"\kappa_{T\otimes_k V}^*"] & \Hom_A(\Omega_A\otimes_A (T\otimes_k V), N) \arrow[d]\arrow[r] & \Ext_A(T\otimes_k V, N) \arrow[d]\arrow[r] & 0\\  
 \Hom_R(\ker{f},N) \arrow[r, "\overline{\kappa_{T\otimes_k V}^*}"]\arrow[d] &  \Hom_A(\Omega_A\otimes_A \ker{f}, N) \arrow[r]\arrow[d] & \Ext_A(\ker{f},N)\arrow[r]& 0. \\
             0                         & 0                                                               & 
\end{tikzcd}
\]

Now we consider the standard projective resolution of $\tilde{M}$ of Theorem \ref{standardRes}
\[
P_\bullet(\tilde{M}) \to \tilde{M} \to 0
\]
and the induced cochain complex $\mathsf{C}:=\Hom_B\big(P_\bullet(\tilde{M}),\tilde{N}\big)$:
\[
\mathsf{C}:\ldots \to 0 \to \Hom_{\tilde{R}}(\tilde{M},\tilde{N}) \to \Hom_A\big(\Omega_A\otimes_A M,N\big)\oplus \Hom_R\big(T\otimes_k V,N\big)
\]
\[
\xto{h^*}\Hom_A\big(\Omega_A\otimes_A (T\otimes_k V),N\big)\to 0 \to ...
\]
Then $H^i(\mathsf{C})=\Ext_{A[T]}^i(\tilde{M},\tilde{N})$ for $i=0,1,2$.

We have
\begin{eqnarray*}
h^* & = &  \Hom_{A[T]}(h,\tilde{N})\\
    & = & \Hom_{A[T]}\left( [\kappa_{T\otimes_k V}\;\;\;\; \Omega_A\otimes f ]^\top ,\tilde{N}\right)\\
    & = & \big[\kappa_{T\otimes V}^*\;\;\;\; (\Omega_A\otimes f)^* \big].
\end{eqnarray*}

Moreover, we have the following commutative diagrams:
\[
\begin{tikzcd}[row sep=large]
T\otimes_k V\arrow[r, "f"]\arrow[d, two heads, "f'"'] & M    \\
\im{f} \arrow[ur, hook , "\iota"']  
\end{tikzcd}   ~  \overset{\Omega_A\otimes_A ?}{\leadsto} ~                             
\begin{tikzcd}[row sep=large]
\Omega_A\otimes_A(T\otimes_k V)\arrow[d, two heads, "\Omega_A\otimes f'"']\arrow[r, "\Omega_A\otimes f"]\arrow[d] & \Omega_A\otimes_A M    \\
 \Omega_A\otimes_A\im{f} \arrow[ur, hook, "\Omega_A\otimes\iota"']  
\end{tikzcd} 
\] 
\[
\overset{\Hom_A(?,N)}{\leadsto}~
\begin{tikzcd}[row sep=large]
\Hom_A(\Omega_A\otimes_A(T\otimes_k V),N)  &\arrow[l, "(\Omega_A\otimes f)^*"'] \Hom_A(\Omega_A\otimes_A M,N) \arrow[dl,two heads, "(\Omega_A\otimes\iota)^*"]   \\
\Hom_A(\Omega_A\otimes_A\im{f},N). \arrow[u,hook, "(\Omega_A\otimes f')^*"]   
\end{tikzcd}                                
\]

Thus we obtain:
\begin{equation} \label{eq:HilfsGleichungx}
\im{(\Omega_A\otimes f')^*} ~~=~~ \im{(\Omega_A\otimes f)^*}.
\end{equation}

We sum up and obtain finally:
\begin{eqnarray*}
\Ext_{A[T]}^2(\tilde{M},\tilde{N}) & = &\Hom_A\big(\Omega_A\otimes_A (T\otimes_k V),N\big) \big/ \im{h^*} \\
							  & = & \Hom_A\left(\Omega_A\otimes_A (T\otimes_k V),N\right) /  \im{[\kappa_{T\otimes V}^*\;\;\;\; (\Omega_A\otimes f)^* ]} \\
							  &\overset{(\ref{eq:HilfsGleichungx})}{=} & \left(\Hom_A\left(\Omega_A\otimes_A (T\otimes_k V),N\right) / \im{(\Omega_A\otimes f')^*} \right) / \im{\overline{\kappa_{T\otimes V}^*}} \\
							  & = & \Hom_A(\Omega_A\otimes_A \ker{f}, N)  /  \im{\overline{\kappa_{T\otimes V}^*}} \\
							  & = &\Ext_A(\ker{f},N).
\end{eqnarray*}

We thus arrive at the folllowing description of $\Ext_{A[T]}^2$:
\begin{thm} \label{Ext2Symmetrie} $ $\\\
For $A[T]$-modules $\tilde{M}=\big(M,V,f:T\otimes_kV\to M\big)$, $\tilde{N}=\big(N,W,g:T\otimes_kW\to N\big)$ there is an isomorphism:
\[
\Ext^2_{A[T]}(\tilde{M},\tilde{N}) ~\xto{\sim}~ \Ext_A(\ker f,N).
\]
\end{thm}

In particular, if $T$ is projective, so is $\ker{f}$, thus ${\rm Ext}^2_{A[T]}$ vanishes identically. In other words, $A[T]$ is again hereditary in this case.

\begin{mydef} We call an $A[T]$-module $\tilde{M}=\big(M,V,f:T\otimes_kV\to M\big)$ \textit{full}, if $f$ is a surjective map.
\end{mydef}

\begin{thm} \label{ext2aufvolle} Assume $\Ext_A(T,T)=0$. Then, for full $A[T]$-modules $\tilde{M}, \tilde{N}$, the vanishing property
\[
\Ext_{A[T]}^2(\tilde{M},\tilde{N}) = 0
\] holds.
\end{thm}
\begin{proof}Write $\tilde{M}=\big(M,V,f:T\otimes_kV\to M\big)$ and $\tilde{N}=\big(N,W,g:T\otimes_kW\to N\big)$.
We consider
\[
0\to \ker{g} \to T\otimes_k W \xto{g}  N \to 0
\]
and apply $\Hom_A(\ker{f},?)$  and obtain the long exact sequence
\[
0\to\Hom_A(\ker{f},\ker{g})\to\Hom_A(\ker{f},T\otimes_k W)\to \Hom_A(\ker{f},N)
\]
\[
\to \Ext_A(\ker{f},\ker{g})\to\Ext_A(\ker{f},T\otimes_k W)\to \Ext_A(\ker{f},N) \to 0.
\]
$ $\\
Now we will show that $\Ext_A(\ker{f},N)=0$ holds by showing that 
\[
\Ext_A(\ker{f},T\otimes_k W)=0
\] 
holds. We also have:
\[
0\to \ker{f} \to T\otimes_k V \xto{f}  M \to 0.
\]
Applying $\Hom_A(?,T\otimes_k W)$ to this, we obtain the long exact sequence:
\[
0\to\Hom_A(M,T\otimes_k W)\to\Hom_A(T\otimes_k V,T\otimes_k W)\to \Hom_A(\ker{f},T\otimes_k W)
\]
\[
\to \Ext_A(M,T\otimes_k W)\to\Ext_A(T\otimes_k V,T\otimes_k W)\to\Ext_A(\ker{f},T\otimes_k W)\to 0.
\]
Since $\Ext_A(T,T)=0$, in particular $\Ext_A(T\otimes_k V,T\otimes_k W)=0$ holds. 

So $\Ext_A(\ker{f},T\otimes_k W)=0$ and therefore $\Ext_A(\ker{f},N)=0$. Using Theorem \ref{Ext2Symmetrie}, we conclude the proof.
\end{proof}

\subsection{Derivations and the Euler form of \texorpdfstring{$A[T]$}{A[T]}}$ $\\
Let $\tilde{M}=\big(M,V,f:T\otimes_kV\to M\big), \tilde{N}=\big(N,W,g:T\otimes_kW\to N\big)$ be $A[T]$-modules. To shorten notation, we define $B=A[T]$. In this section we determine the dimension of a space of derivations
\[
\dim \Der_{\tilde{R}}\big(B,\Hom_{\tilde{R}}(\tilde{N},\tilde{M})\big). 
\]

To do this, we consider the canonical exact sequence
\[
0\to\Hom_B(\tilde{N},\tilde{M})\to \Hom_{\tilde{R}}(\tilde{N},\tilde{M}) \xto c \Der_{\tilde{R}}\big(B,\Hom_{\tilde{R}}(\tilde{N},\tilde{M})\big)
\]
\[\to\Ext_B(\tilde{N},\tilde{M})\to 0,
\]
where $c$ is defined as
\[
\big(\beta_i:\tilde{N}_i\to \tilde{M}_i\big)_{i\in Q_0}
\mapsto \sum_{\substack{\alpha\in Q_1\\ \alpha:i\to j}} 
\tilde{M}_\alpha\beta_i - \beta_j\tilde{N}_\alpha.
\]
We obtain the equality:
\begin{center}
$\dim \Der_{\tilde{R}}\big(B,\Hom_{\tilde{R}}(\tilde{N},\tilde{M})\big)$ \\
$=$ \\
$-\left( \dim \Hom_B(\tilde{N},\tilde{M}) - \dim \Ext_B(\tilde{N},\tilde{M}) \right) + \Hom_{\tilde{R}}(\tilde{N},\tilde{M})$.
\end{center}
$ $\\
On the other hand, we have the following description:

We consider the standard projective resolution of $\tilde{N}$ (Theorem \ref{standardRes})
\[
P_\bullet(\tilde{N})\to \tilde{N } \to 0
\]
and  the induced cochain complex $\mathsf{C}:=\Hom_B\big(P_\bullet(\tilde{N}),\tilde{M}\big)$:
\[
\mathsf{C}:\ldots \to 0 \to \Hom_{\tilde{R}}(\tilde{N},\tilde{M}) \to \Hom_A\big(\Omega_A\otimes_A N,M\big)\oplus \Hom_R\big(T\otimes_k W,M\big)
\]
\[
\xto{h^*}\Hom_A\big(\Omega_A\otimes_A (T\otimes_k W),M\big)\to 0 \to ...
\]
Then $H^i(\mathsf{C})=\Ext_{A[T]}^i(\tilde{N},\tilde{M})~~~(i=0,1,2)$.

This way we obtain the equality:
\begin{center}
$\dim \Hom_B(\tilde{N},\tilde{M})-\dim\Ext_B(\tilde{N},\tilde{M}) + \dim
\Ext^2_B(\tilde{N},\tilde{M})$\\
$=$ \\
$\dim\Hom_{\tilde{R}}(\tilde{N},\tilde{M}) - \dim\Hom_A\big(\Omega_A\otimes_A N,M\big)-\dim\Hom_R\big(T\otimes_k W,M\big) + \dim\Hom_A\big(\Omega_A\otimes_A (T\otimes_k W),M\big)$.
\end{center}
$ $\\
We easily determine:
\begin{eqnarray*}
\dim\Hom_A\big(\Omega_A\otimes_A N,M\big) & = & \sum_{\substack{\alpha\in Q_1\\ \alpha:i\to j}} \dim N_i\cdot\dim M_j, \\
\dim\Hom_R\big(T\otimes_k W,M\big) & = & \dim W\cdot \sum_{i\in Q_0} \dim T_i \cdot \dim M_i, \\
\dim\Hom_A\big(\Omega_A\otimes_A (T\otimes_k W),M\big) & = & \dim W \cdot \sum_{\substack{\alpha\in Q_1\\ \alpha:i\to j}} \dim T_i\cdot\dim M_j.
\end{eqnarray*}
By using the characterization of $\Ext_{A[T]}^2$ of Theorem \ref{Ext2Symmetrie}, we end up in:
\begin{thm}\label{DerivationAllg} For finite-dimensional $A[T]$-modules $\tilde{M}=\big(M,V,f:T\otimes_kV\to M\big)$ and $\tilde{N}=\big(N,W,g:T\otimes_kW\to N\big)$ we have 
\[
\dim \Der_{\tilde{R}}\big(B,\Hom_{\tilde{R}}(\tilde{N},\tilde{M})\big) ~~=~~ \dim \Der_{R}\big(A,\Hom_{R}(N,M)\big)
\]
\[
~+~\dim W \cdot \langle T, M\rangle_A ~~+~~ \dim \Ext_A(\ker{g},M),
\]
where $\langle .,.\rangle_A$ denotes the homological Euler-form of $A$.
\end{thm}
\begin{crl}\label{Derivation}$ $\\If we assume $\Ext_A(T,T)=0$, then for  full $A[T]$-modules $\tilde{M}$ and $\tilde{N}$ we have
\[
\dim \Der_{\tilde{R}}\big(B,\Hom_{\tilde{R}}(\tilde{N},\tilde{M})\big) ~~=~~ \dim \Der_{R}\big(A,\Hom_{R}(N,M)\big) + \dim W \cdot \langle T, M\rangle_A.
\]
\end{crl}
\begin{proof}
The claim follows immediately using Theorem \ref{DerivationAllg} and \ref{ext2aufvolle}.
\end{proof}
We notate a dimension vector $\vec{d}\in\cN^{\hat{Q}_0}$ as a tuple $(s,d)$, where $s=\vec{d}_{\infty}$ and $d\in\cN^{Q_0}$.
\begin{mydef}For dimension vectors $(s,d),(s',d')\in\cN^{\hat{Q}_0}$ we set
\[
\langle (s,d),(s',d')\rangle_{A[T]} := ss' - s\cdot\langle \dimv T, d'\rangle_Q
+\langle d,d' \rangle_Q.
\]
\end{mydef}
\begin{crl}\label{eulerform_folgerung} $ $\\
For the homological {Euler}-form of $A[T]$ the following identity holds: 
\[
\langle \tilde{M},\tilde{N}\rangle_{A[T]} ~=~ \langle \dimv\tilde{M},\dimv\tilde{N}\rangle_{A[T]},
\]
where $\tilde{M},\tilde{N}$ are (finite-dimensional) $A[T]$-modules and $\langle .,.\rangle_Q$ denotes the {Euler}-form of $Q$.
\end{crl}
\begin{proof}
The identity follows from the above discussion.
\end{proof}

\section{Varieties of representations of one-point extensions}
For all standard notions on varieties of representations of algebras, we refer to \cite{LB}. Let $(s,d)$ be a dimension vector of $\hat{Q}$. Using the isomorphism $A[T]\simeq k\hat{Q}/(\Rel)$, we can realize the 
variety of representations of $A[T]$ with dimension $(s,d)$ as a ({Zariski}-)closed subvariety of the variety of representations of $\hat{Q}$ with dimension $(s,d)$ denoted by $\Rep_{(s,d)}(\hat{Q})$:  
\[
\Rep_{(s,d)}(A[T]) \overset{\mathrm{clsd}}\subseteq \Rep_{(s,d)}(\hat{Q}) .
\] 

\subsection{The Zariski-tangent space}$ $\\
For $\tilde{M}\in\Rep_{(s,d)}(A[T])$ we have (see \cite[Example 3.10.]{LB}):
\[
\mathsf{T}_{\tilde{M}}\Rep_{(s,d)}(A[T])~=~\Der_{\tilde{R}}\big(A[T],\End_{\tilde{R}}(\tilde{M})\big). 
\]
We set $t:=\dimv \,T$ and get by using Theorem \ref{DerivationAllg}:

\begin{thm}\label{tangentspace} For $\tilde{M}=\big(M,V,f:T\otimes_kV\to M\big)\in\Rep_{(s,d)}(A[T])$ we have
\[ 
\dim \mathsf{T}_{\tilde{M}}\Rep_{(s,d)}\big(A[T]\big)~=~ \dim\Rep_d(Q) + s\cdot\langle t,d\rangle_Q + \dim \Ext_A(\ker{f},M).
\]
\end{thm}

\subsection{A well-behaved subvariety in the rigid case}$ $\\
We consider 
\[
\Rep_{(s,d)}\voll\big(A[T]\big)~:=~ \bigg\lbrace (f,M)\;:\; M\in\Rep_d(Q) ,\; \, \rank{f}=\gamma_{T^s,d} \bigg\rbrace
\]
\[
\overset{\text{open}}{\subseteq} \, \Rep_{(s,d)}\big(A[T]\big),
\]
where $\gamma_{T^s,d}\in\cN^{Q_0}$ denotes the \textit{general rank} of homomorphism from $T^s$ to a representation of dimension $d$ (see \cite[Section 5]{S}).

\begin{thm}\label{g1Thm1Version}
If we assume $\gamma_{T^s,d}=d$ and $\Ext_A(T,T)=0$, then $\Rep_{(s,d)}\voll\big(A[T]\big)$ 
is smooth and every component has dimension
\[
\dim \Rep_d(Q) + s\cdot\langle \dimv\, T,d\rangle_Q,
\]
i.e. $\Rep_{(s,d)}\voll\big(A[T]\big)$ is a local complete intersection. 
\end{thm}
\begin{proof} By counting the explicit defining polynomial equations 
\[
\Rep_{(s,d)}(A[T]) \overset{\mathrm{clsd}}{\subseteq}\Rep_{(s,d)}(\hat{Q}),
\]
induced by (\ref{eq:relationGabriel}), we find that every (nonempty) component of $\Rep_{(s,d)}(A[T])$ has dimension $\geq \dim\,\Rep_d(Q)+s\cdot\langle \dimv \,T,d\rangle_Q$. 

On the other hand, for each $\tilde{M}\in\Rep\voll_{(s,d)}(A[T])$ we have
\[
\dim_{\tilde{M}}\Rep\voll_{(s,d)}(A[T]) \leq \dim \mathsf{T}_{\tilde{M}}\Rep_{(s,d)}(A[T]).
\]
The claim follows immediately by using the dimension formula in Theorem \ref{tangentspace}.
\end{proof}

\subsection{On homomorphisms from a fixed representation}$ $\\
We denote the open dense subset 
\[
\bigg\{ M\in \Rep_d(Q)\;:\; \dim\Hom_A(T,M)=\hom(T,d) ~~\text{und}~~  \gamma_{T,M}=\gamma_{T,d} \bigg\}
\]
of $\Rep_d(Q)$ by $\mathcal{V}_{T,d}$. Here, $\hom(T,d)$ is the dimension of the space of homomorphisms from the fixed representation $T$ to a general representation of dimension vector $d$ (see \cite{CB4}). Moreover, $\gamma_{T,M}$ is the unique maximal rank of homomorphisms from $T$ to $M$. 

We consider the regular map
\[
\tilde{\pi}:\Rep_{(1,d)}\voll\big(A[T]\big) \to \Rep_d(Q),\,(f,M)\mapsto M,
\]
which induces a regular map
\[
\pi:\mathcal{V} \to \mathcal{V}_{T,d},
\]
where $\mathcal{V}$ is the open preimage $\tilde{\pi}\inv(\mathcal{V}_{T,d})\subseteq \Rep_{(1,d)}\voll\big(A[T]\big)$.

Obviously $\mathcal{V}_{T,d}$ is irreducible and the fibers of $\pi$ are irreducible of dimension $\mathrm{hom}(T,d)$. So there is a unique irreducible component $\mathcal{V}_0$ of $\mathcal{V}$ of maximal dimension which dominates $\pi$, that is, we have $\overline{\pi(\mathcal{V}_0)} = \mathcal{V}_{T,d}$ (Proposition \ref{faserDimIrredKrit})and 
\[
\dim \mathcal{V} = \dim \mathcal{V}_0=\dim \Rep_d(Q) + \mathrm{hom}(T,d). 
\] 

Since the regular points in $\mathcal{V}$ form an open dense subset, there is regular point $\tilde{M}$ in the irreducible component $\mathcal{V}_0$, and we have:
\begin{eqnarray*}
\dim\mathcal{V}_0  & =  & \dim_{\tilde{M}}\mathcal{V} \\
				   & =  & \dim \mathsf{T}_{\tilde{M}}\mathcal{V} ~=~\dim \mathsf{T}_{\tilde{M}}\Rep_{(1,d)}\voll\big(A[T]\big).
\end{eqnarray*}

Using the description of the tangent space by derivations, we thus find:
\begin{thm}\label{schofieldCBRefindOne}We have
\[
\hom(T,d)~=~\langle t,d\rangle_Q ~+~ \dim \Ext_A(\ker{f},M)
\]
for a general $A$-module homomorphism $f:T\to M$. 
\end{thm}

As a side remark, we rediscover the following result of {Schofield} \cite{S} and {Crawley-Boevey} \cite{CB4} : 
\begin{crl}\label{schofieldCBRefindTwo}
We have
\[
\hom(T,d)~=~\langle \gamma_{T,d},d\rangle_Q ~+~ \dim \Hom_A(\ker{f},M)
\]
for a general $A$-module homomorphism $f:T\to M$.
\end{crl}
\begin{proof}Use the the {Euler} form of $Q$ and the identity
\[
\dim\Hom_A(\ker{f},M) ~- ~\dim \Ext_A(\ker{f},M) ~~=~~ \langle \dimv\,\ker{f}, d\rangle . 
\]
\end{proof}
\begin{crl}\label{SchofieldHomFormel}$ $\\
If we assume $\gamma_{T,d}=d$ and $\Ext_A(T,T)=0$, then $\hom(T,d)~=~\langle t,d\rangle_Q$.
\end{crl}

\subsection{An irreducible component in the rigid case}$ $\\
We need the following facts from algebraic geometry:
\begin{prp}\label{dichteTeilmengeWotangentFaserKern}
Let $\pi:\mathcal{X}\to\mathcal{Y}$ be a regular map of affine varieties. Then there is an open dense subset $\mathcal{U}\subseteq \mathcal{X}$ such that for all $x\in\mathcal{U}$ 
\[
\mathsf{T}_x\pi\inv\big(\pi(x)\big) = \ker{\mathsf{d}\pi_{x}}.
\]
\end{prp}
\begin{proof}
See for example \cite{HP}.
\end{proof}

\begin{prp}\label{faserDimIrredKrit}
Let $f:\X\to\Y$ be a regular map of quasi-projective varieties.\\
If $\Y$ is irreducible and all fibers of $f$ are irreducible and of same dimension $d$ (in particular $f$ is surjective), then:
\begin{enumerate}
 \item[a)] There is a unique irreducible component $\X_0$ of $\X$ that dominates $\Y$, i.e. $\overline{f(\X_0)}=\Y$.
 \item[b)]Each irreducible component $\X_i$ of $\X$ is a union of fibers of $f$. Its dimension is equal to $\dim\overline{f(\X_i)}+d$.
 \end{enumerate} 
 In particular, we can conclude $\X$ is irreducible if either of the following holds:
 \begin{enumerate}
 \item[i)] $\X$ is equidimensional.
 \item[ii)] $f$ is closed. 
 \end{enumerate}
\end{prp}
\begin{proof}
See \cite{Mu}.
\end{proof}

\begin{thm}\label{homlogieTrifftGeometrie2}$ $\\If we assume $\gamma_{T^s,d}=d$ and $\Ext_A(T,T)=0$, then $\Rep_{(s,d)}\voll\big(A[T]\big)$ is irreducible and smooth of dimension
\[
\dim \Rep_{(s,d)}\voll\big(A[T]\big)~=~ \dim \Rep_d(Q) + s\cdot\langle t,d\rangle_Q .
\]
\end{thm}
\begin{proof} By Theorem \ref{g1Thm1Version}, it remains to prove that $\Rep_{(s,d)}\voll\big(A[T]\big)$ is irreducible. We consider the dominant map
\[
\tilde{\pi}: \Rep_{(s,d)}\voll\big(A[T]\big) \to \Rep_{d}(Q),\tilde{M}=(M,k^s,f)\mapsto M.
\]
For each $\tilde{M}=(M,k^s,f)\in\Rep_{(s,d)}\voll\big(A[T]\big)$ the regular map
\[
\Rep_{d}(Q) \to \Rep_{(s,d)}\big(A[T]\big),\,N\mapsto (N,k^s,0)
\]
induces a split mono for the differential $\tilde{\pi}$ in $\tilde{M}$
\[
\mathsf{d}\tilde{\pi}_{\tilde{M}}:\mathsf{T}_{\tilde{M}}\Rep_{(s,d)}\voll\big(A[T]\big) \to \mathsf{T}_{M}\Rep_{d}(Q),
\]
i.e. $\mathsf{d}\tilde{\pi}_{\tilde{M}}$ is surjective.
For each $\tilde{M}=(M,k^s,f)\in\Rep_{(s,d)}\voll\big(A[T]\big)$ we have
\[
\Hom_A(T^s,M) ~~ \simeq ~~ \mathsf{T}_{\tilde{M}}\tilde{\pi}\inv\big(\tilde{\pi}(\tilde{M})\big).
\]
There is an open dense subset $\mathcal{U}\subseteq \Rep_{(s,d)}\voll\big(A[T]\big)$ such that for each $\tilde{M}\in\mathcal{U}$ we have
\[
\Hom_A(T^s,M)~\simeq ~\ker{\mathsf{d}\tilde{\pi}_{\tilde{M}}}
\] 
(Theorem \ref{dichteTeilmengeWotangentFaserKern}). Since $\mathsf{d}\tilde{\pi}_{\tilde{M}}$ is surjective, by using Theorem \ref{SchofieldHomFormel} we obtain:
\[
\dim\Hom_A(T^s,M)~=~\hom(T^s,d)
\]
for each $\tilde{M}\in\mathcal{U}$.
So
\[
\mathcal{U}~~\subseteq ~~\bigg\{ (M,k^s,f)\in\Rep_{(s,d)}\voll(A[T])\;:\; \hom(T^s,M)=\dim\Hom_A(T^s,M) \bigg\}=:\mathcal{V}.
\]
Since $\mathcal{U}\subseteq\Rep_{(s,d)}\voll(A[T])$ is dense,
$\mathcal{V}\subseteq\Rep_{(s,d)}\voll(A[T])$ is dense, too.

The map $\tilde{\pi}$ induces a surjective regular map
\[
\pi:\mathcal{V}\to\mathcal{V}_{T^s,d},(M,k^s,f)\mapsto M,
\]
where $\mathcal{V}_{T^s,d}\subseteq\Rep_d(Q)$ is dense (and thus irreducible). {All} fibers of $\pi$ are irreducible of equal dimension and $\mathcal{V}$ is equidimensional, thus $\mathcal{V}$ has to be irreducible by Theorem \ref{faserDimIrredKrit}.
Using $\overline{\mathcal{V}}=\Rep_{(s,d)}\voll(A[T])$, we can conclude the claim.  
\end{proof}


\section{Semistability}
For all notions concerning stability and moduli spaces of representations we refer to \cite{Re}. For an $A[T]$-module $\tilde{M}=(M,V,f:T\otimes_k V\to M)$ we define its slope
\[
\mu\big(\tilde{M}\big)~=~\mu\big(M,V,f\big):=~ \frac{\dim V }{\dim V + \dim M}.
\]
\begin{mydef}
An  $A[T]$-module $\tilde{M}=(M,V,f:T\otimes_k V\to M)$ is called \textit{semi-stable} (resp. \textit{stable}) if 
$\mu(U)\leq \mu(\tilde{M})$ (resp. $\mu(U)<\mu(\tilde{M})$) for all proper subrepresentation $0\neq U\subset\tilde{M}$.
\end{mydef}

\subsection{A first criterion}
\begin{thm}\label{kingStabKoecherErw}
Let $\tilde{M}=(M,V,f:T\otimes_k V\to M)$ be an $A[T]$-module with $\tilde{M}_\infty\neq 0$. Then $\tilde{M}$ is (semi-)stable iff for all subspaces $W\subseteq \tilde{M}_\infty$ the following inequality is fulfilled:
\[
\dim M \cdot \frac{\dim W}{\dim V} \underset{(=)}{<}
\sum_{i\in Q_0} \dim \left(\sum_{l=1}^{t_i}\tilde{M}_{\rho_{l, (i)}}(W)\right).
\] 
\end{thm}
\begin{proof}
For all subobjects $\tilde{U}=(U,W,g)\subsetneq \tilde{M}$, we have:
\[ \mu(\tilde{U})  \underset{(=)}{<}  \mu(\tilde{M}) \Leftrightarrow \dim M \cdot \frac{\dim W}{\dim V} \underset{(=)}{<}  \dim U.   \]
We can easily determine the total space of the smallest subobject $\overline{W}$ of $\tilde{M}$ containing a given $W\subset V$. Namely, we have
\[
\overline{W}_i~=~\sum_{l=1}^{t_i}\tilde{M}_{\rho_{l,(i)}}(W)
\]
for each $i\in Q_0$.
\end{proof}

\subsection{An observation on the {Harder-Narasimhan} filtration}$ $\\
At first we recall the notion of {Harder-Narasimhan} filtration.
\begin{mydef}$ $
\begin{enumerate}
\item[a)] 
A dimension vector $(s,d)\in \cN\times\cN Q_{0}$ is called \textit{(semi-)stable} if there is a
(semi-)stable representation of $A[T]$ with dimension vector $(s,d)$. 
\item[b)] A tuple 
\[
(s,d)^*~=~\big( (s^1,d^1),\ldots, (s^r,d^r) \big)~~\in \big(\cN\times \cN Q_{0}\big)^r
\]
of dimension vectors is of \textit{HN-type} if each $(s^l,d^l)$ ($l=1,\ldots,r$) is semi-stable and 
\[
\mu(s^1,d^1) > \ldots > \mu(s^r,d^r).
\] 
\item[c)] A filtration
\[
0 = \tilde{M}_0 \subset \tilde{M}_1 \subset \ldots \subset \tilde{M}_r =\tilde{M}
\] 
of a representation $\tilde{M}$ (of $A[T]$) is called {Harder-Narasimhan} (\textit{HN}) if each quotient $\tilde{M}_{l}/\tilde{M}_{l-1}$ ($l=1,\ldots,r$) is  semi-stable and 
\[
\mu\big(\tilde{M}_{1}/\tilde{M}_{0}\big) > \mu\big(\tilde{M}_{2}/\tilde{M}_{1}\big) > \ldots >  \mu\big(\tilde{M}_{r}/\tilde{M}_{r-1}\big).
\]
\end{enumerate}
\end{mydef}
\begin{prp}$ $\\
Every representation $\tilde{M}$ of $A[T]$ admits a unique {Harder-Narasimhan} filtration.
\end{prp}
\begin{proof}
See \cite{Re}.
\end{proof}

\begin{mydef}For a HN type $(s,d)^*$ we denote by 
\[
\Rep_{(s,d)^*}^{HN}(A[T])\subseteq \Rep_{(s,d)}(A[T])
\] 
the subset of representations whose HN filtration is of type $(s,d)^*$. $\Rep_{(s,d)*}^{HN}(A[T])$ is called \textit{HN stratrum} for the HN type $(s,d)^*$. 
More generally,  we denote by
\[
\Rep_{(s,d)}^{(s,d)^*}(A[T]) \subseteq \Rep_{(s,d)}(A[T])
\]
the subset of representations $\tilde{M}$ possessing a filtration of type $(s,d)^*$, i.e.  there is a chain of subrepresentations
$
0=\tilde{M}_0\subset \tilde{M}_1\subset \ldots \subset \tilde{M}_r=\tilde{M}
$ 
with $\dimv\,\tilde{M}_l/\tilde{M}_{l-1} = (s^l,d^l)$ for $l=1,\ldots,r$.
\end{mydef}

\begin{thm}\label{StrukturMorphismusSurjektivCharakDurchHNFilt} 
Let $\tilde{M}=(M,V,f:T\otimes_{k} V\to M)$ be a representation of $A[T]$ with $V\neq 0$. Then the following are equivalent:
\begin{enumerate}
\item[a)] $f:T\otimes_{k} V \to M$ is surjective.
\item[b)] For the HN filtration of $\tilde{M}$
\[
0=\tilde{M}_0\subset \tilde{M}_1\subset \ldots \subset \tilde{M}_r = \tilde{M}
\] 
we have
\[
\mu(\tilde{M}_r/\tilde{M}_{r-1}) ~\neq ~ 0.
\]
\end{enumerate}
\end{thm}

\begin{proof}
\underline{a) $\Rightarrow$ b):}
We denote $\tilde{U}=\tilde{M}_{r-1}$, which we write as
\[
\tilde{U}=(U,W,g:T\otimes_{k} W\to U).
\]
Since $\tilde{U}\subsetneq \tilde{M}$ is a subrepresentation, we have the following commutative diagram:
\[
\begin{tikzcd}
 T\otimes_{k} V \ar[r, "f"]  & M  \\
 T\otimes_{k} W \ar[r, "g"] \arrow[u, hook] & U \arrow[u, hook]  \\
\end{tikzcd} .
\]

Assume $\mu(\tilde{M}/\tilde{U})=0$, i.e. $\dim V - \dim W =0$. So we obtain $W=V$, and since $f$ is surjective we can conclude from the above commutative diagram that $U$ already equals $M$. In other words, $\tilde{U}=\tilde{M}$,  contradicting our assumption.

\underline{b) $\Rightarrow$ a):}
Assume the structure morphism $f:T\otimes_{k} V\to M$ is not surjective.
Then we can consider the proper subobject $\tilde{U}\subsetneq \tilde{M}$ induced by the commutative diagram
\[
\begin{tikzcd}
\tilde{M}: &  T\otimes_{k} V \ar[r, "f"]  & M  \\
\tilde{U}:\arrow[u, hook] & T\otimes_{k} V \ar[r, "f'"] \arrow[u, hook] & \im{f}. \arrow[u, hook]  \\
\end{tikzcd} 
\] 

Obviously then we have $\mu(\tilde{M}/\tilde{U})=0$, and since $f$ is not surjective, we neither have $\dimv\,\tilde{M}/\tilde{U} = 0$.
So $\tilde{M}/\tilde{U}$ is a semi-stable representation.

Now look at  the HN filtration of $\tilde{U}$. Since the structure morphism of $\tilde{U}$ is surjective we can conclude from the first part of this proof that for the HN filtration of $\tilde{U}=Y_l\supset Y_{l-1}\supset \ldots \supset Y_1\supset Y_0=0$ we have
\[
\mu(Y_l/Y_{l-1})\neq 0.
\]

Since by definition the slope is always $\geq 0$ we can conclude 
\[
\mu(Y_l/Y_{l-1}) > \mu(\tilde{M}/\tilde{U}).
\]
Using the uniqueness of the HN filtration we finally obtain that 
\[
0=Y_0\subset Y_1\subset \ldots Y_l = \tilde{U} \subset \tilde{M}
\]
must be the HN filtration of $\tilde{M}$. But this contradicts our assumption about the HN filtration of $\tilde{M}$. So $f$ has to be surjective.
\end{proof}

Consequences of Theorem \ref{StrukturMorphismusSurjektivCharakDurchHNFilt} are:
\begin{crl} \label{sstSindGenVoll}
For $(s,d)\in \cN\times \cN^{Q_{0}}$ we have the following connection between the semi-stable representations and full representations: 
\[
\Rep_{(s,d)}^\sst(A[T]) ~ \subseteq ~ \Rep_{(s,d)}\voll(A[T]).
\]
\end{crl}

\subsection{Geometric consequences for the moduli space}$ $\\
The linear algebraic group 
\[
\G_{(s,d)}:=\Gl_{s}(k)\times\prod_{i\in Q_0}\Gl_{d_i}(k)
\]
acts on $\Rep_{(s,d)}( \hat{Q})$ via the base change action
\[
(g_i)_i\cdot (\tilde{M}_\alpha)_{\alpha\in\hat{Q}_1}
=
\big( g_j \tilde{M}_\alpha g_i\inv \big)_{\alpha:i\to j}.
\]
$\Rep_{(s,d)}(A[T])$ is stable under this $\G_{(s,d)}$-action. By definition, the $\G_{(s,d)}$-orbits in $\Rep_{(s,d)}(A[T])$ correspond bijectively to isomorphism classes $[\tilde{M}]$ of representations of $A[T]$ of dimension vector $(s,d)$. We consider the stability function $\tilde{\Theta}$ for $\hat{Q}$ given by $\tilde{\Theta}_{\infty}=1$ and $\tilde{\Theta}_i=0$ for $i\in Q_0$. The associated slope function on representations of $\hat{Q}$ coincides with the slope function $\mu$ on $A[T]$-modules. This allows us to define moduli spaces $\M_{(s,d)}^\sst\big(A[T]\big)$ resp. $\M_{(s,d)}^\st\big(A[T]\big)$ as the algebraic quotient of $\Rep_{(s,d)}^\sst(A[T])$, resp. the geometric quotient of $\Rep_{(s,d)}^\st(A[T])$, by $\G_{(s,d)}$.  
\begin{thm}\label{GeoPropModuli} Assume $\Ext_A(T,T)=0$ and $\gamma_{T^s,d}=d$.\\
If $\Rep_{(s,d)}^\st\big(A[T]\big)\neq\emptyset$, then both $\M_{(s,d)}^\sst\big(A[T]\big)$ and $\M_{(s,d)}^\st\big(A[T]\big)$ are irreducible and smooth of dimension
\[
\dim\,\M_{(s,d)}^\sst\big(A[T]\big)~=~ 1 - \langle (s,d),(s,d)\rangle_{A[T]}.
\]
\end{thm}
\begin{proof}
We calculate fibre dimensions for the geometric quotient 
\[
\pi:\Rep_{(s,d)}^\st\big(A[T]\big) \to \M_{(s,d)}^\st\big(A[T]\big),
\]
and use Corollary \ref{sstSindGenVoll}, Theorem \ref{ext2aufvolle} and the fact that the endomorphism rings of stable representations are trivial.
\end{proof}

Next, we introduce the Harder-Narasimhan stratification. Note that the term stratification is used in a weak sense, meaning a finite decomposition of a variety into locally closed subsets.
\subsection{Harder-Narasimhan stratification}\label{hns}$ $\\
In this section we write $\Rep_{(s,d)}$ for $\Rep_{(s,d)}\voll(A[T])$ to simplify notation.
\begin{thm}\label{verAllgKin} Let assume $\Ext_A(T,T)=0$ and $\gamma_{T^s,d}=d$.\\The HN-strata for the HN-types 
\[
(s,d)^* ~=~\big( (s^1,d^1),\ldots, (s^r,d^r) \big)
\] 
with weight $(s,d)$, i.e. $(s,d)=\sum_{i=1}^r(s^i,d^i)$, and $s^r\neq 0$ define a stratification of $\Rep_{(s,d)}$.

The codimension of $\Rep_{(s,d)^*}^{HN}$ in $\Rep_{(s,d)}$ is given by:
\[
- \sum_{k<l}\langle (s^k,d^k),(s^l,d^l)\rangle_{A[T]}.
\] 
\end{thm}
\begin{proof}
Let $F^*:0=F^0\subset F^1\subset \ldots \subset F^r$ be a flag of type $(s,d)^*$ in the $\hat{Q}_0$-graded vector space $k^d\times k^s =\oplus_{i\in Q_0}k^{d_i}~\times~k^{s}$, i.e. $F^{l}/F^{l-1}\simeq k^{d^l}\times k^{s^l}$ for $l=1,\ldots, r$, and denote by $F_i^l$ the $i$-component of $F^l$.

Denote by $\tilde{\mathcal{Z}}_{(s,d)^*}$ the closed subvariety of $\Rep_{(s,d)}(A[T])$ of representations $\tilde{M}$ which are compatible with $F^*$, i.e.  
$\tilde{M}_\gamma(F_i^l)\subset F_j^l$ for $l=1,\ldots, r$ and for all arrows $(\gamma:i\to j)$ in $\hat{Q}$. 
We have the regular map 
\[
\tilde{p}_{(s,d)^*}:\tilde{\mathcal{Z}}_{(s,d)^*} \to \Rep_{(s^1,d^1)}(A[T]) \times \ldots \times \Rep_{(s^r,d^r)}(A[T])
\] 
given by the projection $\tilde{p}_{(s,d)^*}$ mapping $\tilde{M}\in \tilde{\mathcal{Z}}_{(s,d)^*}$ to the sequence of subquotients with respect to $F^*$.
The map $\tilde{p}_{(s,d)^*}$ induces a regular map
\[
p_{(s,d)^*}:\mathcal{Z}_{(s,d)^*}\to \Rep_{(s^1,d^1)}\times \ldots \times\Rep_{(s^r,d^r)},
\]
where
\[
\mathcal{Z}_{(s,d)^*}:=\tilde{p}_{(s,d)^*}\inv\big(\Rep_{(s^1,d^1)}\times \ldots \times \Rep_{(s^r,d^r)}\big) \subseteq \tilde{\mathcal{Z}}_{(s,d)^*}~~~\text{open.}
\]
 
A minute reflection shows that 
\[
\mathcal{Z}_{(s,d)^*}\subset \Rep_{(s,d)},
\] 
and it is a locally closed subset. By Theorem \ref{homlogieTrifftGeometrie2}, 
\[
\Rep_{(s,d)},\Rep_{(s^1,d^1)},\ldots,\Rep_{(s^r,d^r)}~~~\text{are irreducible.}
\] 
We set
\[
e^*_2 ~:=~\big( (s^1,d^1), (s^2,d^2) \big).
\]
We obtain the regular map
\[
p_{e^*_2}: \mathcal{Z}_{e^*_2}  \to \Rep_{(s^1,d^1)}\times \Rep_{(s^2,d^2)},
\]
with fibers at $(\tilde{M}_1,\tilde{M}_2)\in\Rep_{(s^1,d^1)}\times \Rep_{(s^2,d^2)}$ given by
\[
p_{e^*_2}\inv(\tilde{M}_1,\tilde{M}_2)~=~ \Der_{\tilde{R}}\big(B,\Hom_{\tilde{R}}(\tilde{M}_2,\tilde{M}_1)\big).
\]
Thus all fibers are irreducible and of equal dimension (Corollary \ref{sstSindGenVoll} and \ref{Derivation}). By  Theorem \ref{faserDimIrredKrit} we have
\[
\dim\mathcal{Z}_{e^*_2} ~=~\dim p_{e^*_2}\inv(\tilde{M}_1,\tilde{M}_2) + \dim \left( \Rep_{(s^1,d^1)}\times \Rep_{(s^2,d^2)} \right).
\]

Now we set
\begin{eqnarray*}
\overline{e}^*_3  &:= &\big( (s^1,d^1) + (s^2,d^2), (s^3,d^3) \big),\\
e^*_3  &:= &\big( (s^1,d^1), (s^2,d^2), (s^3,d^3) \big).
\end{eqnarray*}
Therefore we obtain the regular map
\[
p_{\overline{e}^*_3}: \mathcal{Z}_{\overline{e}^*_3}  \to \Rep_{(s^1,d^1) + (s^2,d^2)}\times \Rep_{(s^3,d^3)}.
\]

Since
\[
Z_{e^*_3}=p_{\overline{e}^*_3}\inv( \mathcal{Z}_{e^*_2} \times \Rep_{(s^3,d^3)} ),
\]
the map $p_{\overline{e}^*_3}$ induces a regular map 
\[
\mathcal{Z}_{e^*_3} \to \mathcal{Z}_{e^*_2} \times \Rep_{(s^3,d^3)}.
\]
As above, we see that this regular map has irreducible fibers of equal dimension, that is, for  $(\tilde{M}_1,\tilde{M}_3)\in\mathcal{Z}_{e^*_2} \times \Rep_{(s^3,d^3)}$, we have:
\[
\dim \mathcal{Z}_{e^*_3} ~=~ \dim\Der_{\tilde{R}}\big(B,\Hom_{\tilde{R}}(\tilde{M}_3,\tilde{M}_1)\big) + \dim\mathcal{Z}_{e^*_2} + \dim\Rep_{(s^3,d^3)}.
\]

Inductively, we obtain in this way a regular map
\[
\mathcal{Z}_{e^*_r} \to \mathcal{Z}_{e^*_{r-1}} \times \Rep_{(s^r,d^r)}
\]
with irreducible fibers of equal dimension, where
\begin{eqnarray*}
e^*_{r-1}  & = & \big( (s^1,d^1), \ldots, (s^{r-1},d^{r-1}) \big), \\
e^*_{r}  & = & (s,d)^*.
\end{eqnarray*}
Summing up, for $(\tilde{M}_1,\tilde{M}_r)\in\mathcal{Z}_{e^*_{r-1}} \times \Rep_{(s^r,d^r)}$ this yields:
\[
\dim\mathcal{Z}_{(s,d)^*} ~=~ 
\dim\Der_{\tilde{R}}\big(B,\Hom_{\tilde{R}}(\tilde{M}_r,\tilde{M}_1)\big) + \dim\mathcal{Z}_{e^*_{r-1}} + \dim\Rep_{(s^r,d^r)}.
\]
 
Together with the formula in Corollary \ref{Derivation}, we find:
\[
\dim \mathcal{Z}_{(s,d)^*} ~=~ \sum_{n<l} \sum_{\substack{\alpha\in Q_1\\ \alpha: i\to j}} d_i^ld_j^n + s^l \langle t, d^n\rangle_Q ~~+~~\sum_{i=1}^r\dim \Rep_{(s^i,d^i)}.
\]

To simplify the notation in the following, we write $\mathcal{Z}$ (resp. $p$) for $\mathcal{Z}_{(s,d)^*}$ (resp. $p_{(s,d)^*}$).
The preimage of 
\[
\Rep_{(s^1,d^1)}^\sst\times \ldots \times \Rep_{(s^r,d^r)}^\sst
\] 
under $p$ gives us an open subvariety $\mathcal{Z}_0$ of $\mathcal{Z}$ and $\tilde{\mathcal{Z}}$. Since the varieties
\[
\Rep_{(s^1,d^1)},\ldots, \Rep_{(s^r,d^r)}
\] 
are irreducible, we see in a similar manner to $\mathcal{Z}$ that $\dim \mathcal{Z}_0=\dim \mathcal{Z}$ holds. 

The action of $\mathcal{G}_{(s,d)}$ on $\Rep_{(s,d)}(A[T])$ induces actions of the parabolic subgroup $\mathcal{P}_{(s,d)^*}$ of $\mathcal{G}_{(s,d)}$, consisting of elements fixing the flag $F^*$, on $\mathcal{Z}_0$ and $\mathcal{Z}$. The image of the associated fiber bundle $\mathcal{G}_{(s,d)}\times^{\mathcal{P}_{(s,d)^*}} \tilde{\mathcal{Z}}$ under the action morphism $m$ equals $ \Rep^{(s,d)^*}_{(s,d)}(A[T])$, which is thus a closed subvariety of $\Rep_{(s,d)}(A[T])$. The image of $\mathcal{G}_{(s,d)}\times^{\mathcal{P}_{(s,d)^*}} \mathcal{Z}_0$ under $m$ equals $\Rep_{(s,d)^*}^{HN}$, and  $\mathcal{G}_{(s,d)}\times^{\mathcal{P}_{(s,d)^*}} \mathcal{Z}_0$ is the full preimage. By the uniqueness of the HN filtration, the morphism $m$ is bijective over $\Rep_{(s,d)^*}^{HN}$,  which therefore is a locally closed subvariety of $\Rep_{(s,d)}(A[T])$. 

The canonical map $\mathcal{G}_{(s,d)}\times^{\mathcal{P}_{(s,d)^*}} \mathcal{Z}_0 \to \mathcal{G}_{(s,d)}\big/ \mathcal{P}_{(s,d)^*}$ is ({Zariski}) locally trivial. Therefore
\[
\dim \mathcal{G}_{(s,d)}\times^{\mathcal{P}_{(s,d)^*}} \mathcal{Z}_0~~~=~~
(\dim \mathcal{G}_{(s,d)} ~-~ \mathcal{P}_{(s,d)^*})~~+~~\dim \mathcal{Z}_0  .
\]
The codimension of $\Rep_{(s,d)^*}^{HN}$ in $\Rep_{(s,d)}$ is now easily computed as 
\[
- \sum_{n<l}\langle (s^n,d^n),(s^l,d^l)\rangle_{A[T]},
\] 
using the identity $\langle d,d \rangle_Q  =  \dim \mathcal{G}_d - \dim \Rep_d(Q)$ and the above description of $\mathcal{Z}_0$.
\end{proof}

From this description, we can derive a recursive criterion for the existence of semi-stable representations:
\begin{thm}\label{rekursiveFormula} Let us assume $\Ext_A(T,T)=0$. $ $\\
A dimension type $(s,d)\in \cN\times \cN Q_{0}$ is semi-stable if and only if $\gamma_{T^s,d}=d$ and there exists no HN type 
\[
(s,d)^* ~=~\big( (s^1,d^1),\ldots, (s^r,d^r) \big)
\] 
with weight $(s,d)$ and $s^r\neq 0$ such that  
\[
~~~~~~~~~~~~~~~~~~~~~~~~ \langle (s^n,d^n),(s^l,d^l)\rangle_{A[T]} ~=~0 \;\;\;\;\;\text{(for all $1\leq n<l\leq r$)}.
\]
\end{thm}
\begin{proof}
Let 
\[
H:=\left\{ (s,d)^* \,:\, \exists\, r\in\cN_{>1},\,\text{$(s,d)^* = \big( (s^1,d^1),\ldots, (s^r,d^r) \big)$ is of HN-type} \right.
\]
\[
\left. \text{with weight $(s,d)$ and $s^r\neq0$}  \right\}.
\]
Obviously $H$ is a finite set. Using Theorem \ref{StrukturMorphismusSurjektivCharakDurchHNFilt} we get 
\[
\Rep_{(s,d)}~~=~~ \Rep_{(s,d)}^\sst ~\sqcup \bigsqcup_{(s,d)^*\in H} \Rep^{HN}_{(s,d)^*}. 
\]

If $\Rep_{(s,d)}^\sst\neq \emptyset$, $\Rep_{(s,d)}^\sst$ is of equal dimension like $\Rep_{(s,d)}$.  So 
\[
\mathrm{codim}_{\Rep_{(s,d)}}\Rep^{HN}_{(s,d)^*} > 0
\]
for $(s,d)^*\in H$ implies $\Rep_{(s,d)}^\sst\neq\emptyset$.

Let $n<l$. Since $(s^n,d^n)$ resp. $(s^l,d^l)$ are semi-stable, we find semi-stable representations $\tilde{N}$ resp. $\tilde{M}$  
of $A[T]$ with $\dimv\,\tilde{N}=(s^n,d^n)$ resp. $\dimv\,\tilde{M}=(s^l,d^l)$ and by using Corollary \ref{ext2aufvolle} and \ref{sstSindGenVoll} we get the relation  
\[
\langle (s^n,d^n),(s^l,d^l)\rangle_{A[T]}~~=~~ \dim\Hom_{A[T]}(\tilde{N},\tilde{M}) ~-~ \dim \Ext_{A[T]}(\tilde{N},\tilde{M}).
\]

From $\mu(s^n,d^n) > \mu(s^l,d^l)$ we deduce $\Hom_{A[T]}(\tilde{N},\tilde{M})=0$, in particular
\[
\langle (s^n,d^n),(s^l,d^l)\rangle_{A[T]}~~=~~ ~-~ \dim \Ext_{A[T]}(\tilde{N},\tilde{M}).
\]
So the equation
\[
\sum_{n<l}\langle (s^n,d^n),(s^l,d^l)\rangle_{A[T]} ~~=~~0
\]
is fulfilled if and only if 
\[
\langle (s^n,d^n),(s^l,d^l)\rangle_{A[T]} ~=~0
\]
for all $1\leq n<l\leq r$.
\end{proof}

\begin{bsp}\label{BeispielWeiter}We carry on with Example \ref{BeispielStart} here. Obviously the $\mathcal{G}_{(3,1)}$-orbit of $T = (k^3 \xto{ [1~0~0] }k)$ is dense in $\Rep_{(3,1)}(Q)$. Therefore we have $\Ext_{kQ}(T,T)=0$.
\begin{enumerate}
    \item[1)]
    Applying the recursive criterion we derive that $(2,4,1)$ is semi-stable. In fact, by using the first criterion in Theorem \ref{kingStabKoecherErw}, we see that in this case the semi-stability notion equals to the stability notion. From the geometry of the moduli (Theorem \ref{GeoPropModuli}) we can conclude $\dim \M_{(2,4,1)}^\st(A[T]) = 4$.
    \item[2)] Analogous statements as in 1) hold for the dimension vector $(3,6,2)$. And we can deduce $\dim\,\M_{(3,6,2)}^\st\big(A[T]\big) = 6$.
\end{enumerate}
\end{bsp}

\section{Generating semi-invariants}
In this section we assume that $Q$ is an acyclic quiver. We determine a set of functions generating the ring of semi-invariants for given dimension vector $(s,d)\in\cN \hat{Q}_0$ on $\Rep_{(s,d)}(A[T])$ under the $\G_{(s,d)}$ base change action.

We start with a general observation:
\begin{lmm}\label{inducedInv}
Let $\G$ be a linear reductive group and $\mathcal{X}$ an affine $\G$-variety. Let $\mathcal{A}\subseteq \mathcal{X}$ be a closed and $\G$-stable subset. Then, for every semi-invariant function $f:\mathcal{A}\to k$ there is a semi-invariant $\tilde{f}:\mathcal{X}\to k$ such that $\tilde{f}\vert_{\mathcal{A}} = f$ holds. 
\end{lmm}
\begin{proof} Let $\chi$ be a character of $\G$.
We consider the action of $\G$ on $\mathcal{X}\times k$ given by 
\[
g.(x,\lambda)~:=~(gx, \chi(g)\lambda),
\]
where $(x,\lambda)\in \mathcal{X}\times k,\,g\in \G$.

Since $\mathcal{A}\times k ~\subseteq~ \mathcal{X}\times k$ is closed and $\G$-stable, the categorical quotient $(\mathcal{A}\times k)// \G ~\subseteq~ (\mathcal{X}\times k)//\G$ is closed, i.e.
\[
k[\mathcal{X}\times k]^{\G} \to k[\mathcal{A}\times k]^{\G} ,\, f \mapsto f\vert_{\mathcal{A}\times k}
\]  
is surjective.
\end{proof}

In the following, to simplify the notation we denote by $\hat{A}$ the path algebra of the one-point extended quiver $\hat{Q}$.
 
Let $\tilde{N}$ be a representation of $\hat{Q}$ with projective resolution
\[
0\to \bop_{v\in \hat{Q}_0} \hat{A}e_v^{b(v)} \xto\theta \bop_{v\in \hat{Q}_0}\hat{A}e_v^{a(v)} \to \tilde{N}\to 0.
\]
For $\tilde{M}\in\Rep_{(s,d)}(\hat{Q})$ we apply the functor $\Hom_{\hat{A}}(?,\tilde{M})$ to this resolution and get
\[
0\to \Hom_{\hat{A}}(\tilde{N},\tilde{M})\to \Hom_{\hat{A}}(\bop_{v\in \hat{Q}_0}\hat{A}e_v^{a(v)},\tilde{M}) \xto{ \Hom_{\hat{A}}(\theta,\tilde{M})}\]
\[
\Hom_{\hat{A}}(\bop_{v\in \hat{Q}_0}\hat{A}e_v^{b(v)},\tilde{M})
\to\Ext_{\hat{A}}(\tilde{N},\tilde{M})\to 0.
\]
The condition $\langle \dimv\,\tilde{N}, (s,d)\rangle_{\hat{Q}}=0$ is equivalent to
\[
\sum_{v\in \hat{Q}_0}a(v)\hat{d}_v~=~\sum_{v\in \hat{Q}_0}b(v)\hat{d}_v,
\]
that is, in this case we end up with a linear map between vector spaces of equal dimension. {Schofield} and {Van den Bergh} proved (\cite{S2}) that all semi-invariant\\functions arise as linear combination of functions
\[
\omega_{\tilde{N}}:\Rep_{(s,d)}(\hat{Q})\to k,\, \tilde{M}\mapsto\det(\Hom_{{\hat{A}}}(\theta,\tilde{M}))
\]
induced by representations $\tilde{N}$ of $\hat{Q}$ with $\langle \dimv\,\tilde{N}, (s,d)\rangle_{\hat{Q}}=0$.\\
Using the relation $\hat{A}\epi A[T]$ we can conclude with Lemma \ref{inducedInv}:
\begin{lmm}$ $\\The functions $\omega_{\tilde{N}}$ for representations $\tilde{N}$ of $A[T]$ such that $\langle \dimv\,\tilde{N},(s,d)\rangle_{\hat{Q}}=0$
generate the ring of invariants on $\Rep_{(s,d)}(A[T])$.
\end{lmm}

We can improve this description further by using the canonical exact sequence for one-point extensions and the explicit description of the standard projective resolution in Theorem \ref{standardRes}: Let $\tilde{N}$ be a representation of $A[T]$. Then we have the canonical exact sequence 
\[
\begin{tikzcd}
{0} \arrow[r]  & \tilde{N}_{\mathrm{full}} \arrow[r] & \tilde{N}  \arrow[r] &  \tilde{N}_{0}  \arrow[r] &  {0}   \end{tikzcd}\]
which in more detail reads
\[\begin{tikzcd}
0\arrow[r]  & T\otimes_k W \arrow[r] \arrow[d] & T\otimes_k W  \arrow[r]\arrow[d]&  0 \arrow[d] \arrow[r] & 0   \\
0 \arrow[r]  & \im{f} \arrow[r] & N  \arrow[r]&  \coker{f}   \arrow[r]&  0.   
\end{tikzcd}  
\]
Obviously $\tilde{N}_{0}$ can be interpreted as a representation of $Q$, with standard resolution of length $1$. Thus, we arrive at the following commutative diagram with exact columns and rows:
\[
\begin{tikzcd}
& \bop_{v\in \hat{Q}_0} Be_v^{b(v)}\ar[d,"\varphi_1"]\ar[r,"\theta"] &  \bop_{v\in \hat{Q}_0}Be_v^{a(v)}  \ar[r] \ar[d,"\varphi_0"] & \tilde{N} \ar[d] \ar[r] & 0\\
0\ar[r]& \bop_{v\in \hat{Q}_0} Be_v^{c(v)}\ar[r,"\rho"] \ar[d] & \bop_{v\in \hat{Q}_0}Be_v^{e(v)} \ar[r]\ar[d] & \tilde{N}_{0} \ar[r] \ar[d] & 0. \\
& 0 & 0 & 0
\end{tikzcd}
\]
Thus, if $\langle \dimv\,\tilde{N},(s,d)\rangle_{\hat{Q}}=0$ and  $\omega_{\tilde{N}} \neq 0$ hold, we can conclude from the diagram that the determinants $\omega_{\tilde{N}_{0}}$ and $\omega_{\tilde{N}_{\mathrm{full}}}$ can be formed. This discussion shows:
\begin{thm} The ring of semi-invariant functions on $\Rep_{(s,d)}(A[T])$ is generated by the functions $\omega = \omega_{L} \cdot \omega_{\tilde{N}}$ induced by representations $L$ of $Q$ such that $\langle \dimv \,L,d\rangle_{Q}=0$ and full representations $\tilde{N}$ of $A[T]$ such that $\langle \dimv\,\tilde{N},(s,d)\rangle_{\hat{Q}}=0$.
\end{thm}

From the homological properties we can further conclude:
\begin{thm} For a character $\chi$ of $\G_{(s,d)}$, a representation $\tilde{M}=(M,V,f:T\otimes_kV\to M)\in\Rep_{(s,d)}(A[T])$ is $\chi$-semi-stable iff there is a non-trivial finite-dimensional representation $\tilde{N}=(N,W,g:T\otimes_kW\to N)$ of $A[T]$ such that 
\begin{itemize}
\item[a)] $\Hom_{A[T]}(\tilde{N},\tilde{M}) = 0$, and
\item[b)] $\langle \dimv \tilde{N},(s,d)\rangle_{A[T]} = \dim W \cdot \sum_{\substack{\alpha\in Q_1\\ \alpha:i\to j}} \dim T_i \cdot d_j$.
\end{itemize}
\end{thm}
\begin{proof} Take the standard resolution as described as in Theorem \ref{standardRes}
\[
P_\bullet(\tilde{N})\to \tilde{N } \to 0
\]
and consider the cochain complex $\mathsf{C}:=\Hom_{A[T]}\big(P_\bullet(\tilde{N}),\tilde{M}\big)$:
\[
\mathsf{C}:0 \to \Hom_{\tilde{R}}(\tilde{N},\tilde{M}) \xto{\hat{\phi}} \Hom_A\big(\Omega_A\otimes_A N,M\big)\oplus \Hom_R\big(T\otimes_k W,M\big)
\]
\[
~~~~~~~~~~~~~~~~~~~~~~~~~~~~~~~~~~~~~~~~~~~~~~~~~~~~~~~~~~~~~~~~~~~~~~~~~~~~~~\to\Hom_A\big(\Omega_A\otimes_A (T\otimes_k W),M\big)\to 0. 
\]
Then you have $H^i(\mathsf{C})=\Ext_{A[T]}^i(\tilde{N},\tilde{M})~~~(i=0,1,2)$.

In this way, we achieve the relation:
\begin{center}
$\dim \Hom_{A[T]}(\tilde{N},\tilde{M})-\dim\Ext_{A[T]}(\tilde{N},\tilde{M}) + \dim
\Ext^2_{A[T]}(\tilde{N},\tilde{M})$\\
$=$ \\
$\dim\Hom_{\tilde{R}}(\tilde{N},\tilde{M}) - \dim\Hom_A\big(\Omega_A\otimes_A N,M\big)-\dim\Hom_R\big(T\otimes_k W,M\big)$\\
$~~~~~~~+ \dim\Hom_A\big(\Omega_A\otimes_A (T\otimes_k W),M\big)$.
\end{center}

Both conditions in the theorem are equivalent to $\hat{\phi}$ being an isomorphism.
\end{proof}

\begin{bsp}\label{BeispielWeiterWeiter}We carry on with Example \ref{BeispielWeiter} here. Long calculations yields to following generating and algebraic independent semi-invariant functions:
\begin{enumerate}
    \item[1)] The regular maps $\Rep_{(2,4,1)}\voll(A[T]) \to k$ defined by
$$
\hbar_0: (A,B,C,M)  \mapsto   -\det \left(
\begin{array}{c|c|c}
A & B & C\\
\hline 
0 & \begin{pmatrix}MA \\ 0 \end{pmatrix} &
 \begin{pmatrix}0 \\ MA \end{pmatrix}
\end{array}
\right),
$$
$\hbar_1: (A,B,C,M) \mapsto \det(A\,\vert \, B)$, $\hbar_2: (A,B,C,M) \mapsto \det(A\,\vert \, C)$, \\ 
$\hbar_3: (A,B,C,M) \mapsto \det(A+C\,\vert\,B)$, $\hbar_4: (A,B,C,M) \mapsto \det(A\,\vert\,B+C)$, and $\hbar_{5}: (A,B,C,M) \mapsto \det(A+C\,\vert \, B+C)$
give rise to the geometric quotient $\hat{\pi}:\Rep_{(2,4,1)}^\st(A[T]) \to \cP^4$ by $\G_d$ in the following way 
\[
\hat{\pi}:=\big(\hbar_0\hbar_1,\hbar_0\hbar_2,\hbar_0\hbar_3,\hbar_0\hbar_4,\hbar_0\hbar_5\big).
\]
Thus, we have $\M_{(2,4,1)}^\st\big(A[T]\big) ~=~\cP^4$.

    \item[2)]In the case of the dimension vector $(3,6,2)$ the regular maps\\ $\Rep_{(3,6,2)}\voll(A[T]) \to k$ defined by
    
    $$
   \hbar_0: (A,B,C,M) \mapsto  
    \det{\left(
    \begin{array}{c|c|c||c|c|c}
    MA &   &  &   & -MA &   \\
    \hline
    & MA &  &  &  & -MA \\
    \hline
    &  &  MA & & & -MA \\
    \hline
    \hline
    A &   & C &  & B & \\
    \hline
    & B &   &  A & & C 
    \end{array}\right)}, \\ 
    $$
    $\hbar_1: (A,B,C,M) \mapsto  \det(A\,\vert \, B)$, $\hbar_2: (A,B,C,M) \mapsto \det(A\,\vert \, C)$,\\
    $\hbar_3: (A,B,C,M) \mapsto \det(A+C\,\vert\,B)$, $\hbar_4: (A,B,C,M) \mapsto \det(A+B\,\vert\,C)$, \\
    $\hbar_5: (A,B,C,M) \mapsto \det(A\,\vert\,B+C)$, $\hbar_6: (A,B,C,M)  \mapsto \det(A+B\,\vert\,B+C)$, and $\hbar_{7}: (A,B,C,M)\mapsto \det(A+C\,\vert \, B+C)$ give rise to the geometric quotient $\hat{\pi}:\Rep_{(3,6,2)}^\st(A[T]) \to \cP^6$ by $\G_d$ in the following way 
\[
\hat{\pi}:=\big(\hbar_0\hbar_1,\hbar_0\hbar_2,\hbar_0\hbar_3,\hbar_0\hbar_4,\hbar_0\hbar_5, \hbar_0\hbar_6, \hbar_0\hbar_7\big),
\]
This shows that $\M_{(3,6,2)}^\st\big(A[T]\big) ~=~\cP^6$.
\end{enumerate}
\end{bsp}

\section{Higher Gel'fand MacPherson correspondence}
We first recall the definition of quiver Grassmannians (see for example \cite{CFR1}). For a quiver $Q$, a representation $X$ of $Q$ of dimension vector $d$ and another dimension vector $e\leq d$, we define $\Gr^e_Q(X)$ as the set of subrepresentations $U$ of $X$ of dimension vector $d-e$. This carries a natural scheme structure as the geometric quotient by the base change group $\G_d$ of the set ${\rm Hom}_Q(X,e)_e$ of surjections (that is, rank $e$ maps) from $X$ to a representation of dimension vector $e$.

From now on, we assume $\End_A(T)=k$, $\Ext_A(T,T)=0$ and $\gamma_{T^s,d}=d$.

In this case, the regular map
\[
\phi:\Rep_{(s,d)}\voll(A[T]) \to \Hom_Q(T^s,d)_{d},
\big(M,f:T^s\to M\big)\mapsto f,
\]
which is always a locally trivial fiber bundle (\cite[Lemma 1.2.]{CB4}), has single element fibres, and thus is an isomorphism of varieties. Moreover, we have the $\G_d$-bundle
\[
\kappa: \Hom_Q(T^s,d)_{d} \to \Gr^d_Q(T^s),(\theta_i)_{i\in Q_0}\mapsto \ker{\oplus_{i\in Q_0}\theta_i}.
\]

All in all, we achieve a $\G_d$-bundle
\[
\psi: \Rep_{(s,d)}^\sst(A[T]) \to \Gr^{d,\sst}_Q(T^s),
\]
where $\Gr^{d,\sst}_Q(T^s):=\kappa\circ\phi\bigg(\Rep_{(s,d)}^\sst(A[T])\bigg)\subseteq \Gr^{d}_Q(T^s)$ is open.

We thus have an induced map
\[
\pi:\Gr^{d,\sst}_Q(T^s) \to \M_{(s,d)}^\sst\big(A[T]\big).
\]

The linear reductive group $\Aut_Q(T^s)$ acts naturally on $\Gr_Q^d(T^s)$ and for $\big[\tilde{M}\big]\in\M_{(s,d)}^\sst\big(A[T]\big)$ clearly we have
\[
\pi\inv\big([\tilde{M}]\big)=\Aut_Q(T^s).[(M,f)].
\]

We thus find:

\begin{thm}[Higher {Gel'fand MacPherson} correspondence]\label{gelfand_mac}$ $\\
There is an isomorphism of varieties:
\[
\M^\sst_{(s,d)}\big(A[T]\big)~\simeq~\Gr^{d,\sst}_Q(T^s) \big/\Aut_Q(T^s).
\]
\end{thm}
\begin{proof}By Theorem \ref{GeoPropModuli} $\M^\sst_{(s,d)}\big(A[T]\big)$ is smooth, in particularly normal. Theorem \ref{homlogieTrifftGeometrie2} shows that the quiver {Grassmanian} of a representation without self-extensions is irreducible. Since $\pi$ is surjective the claim follows from \cite[Theorem 4.2]{PV}.
\end{proof}	

\section{Motive of the moduli space}
In this section, we assume $\Ext_A(T,T)=0$ and $\gamma_{T^s,d}=d$.
 
As an application of the arguments in Theorem \ref{verAllgKin} and the explicit recursive formula given there, we derive a formula for the motive of the (smooth) moduli space $\M_{(s,d)}^\sst(\cC):=\M_{(s,d)}^\sst\big(A[T]\big)$ (over $\cC$).

To achieve this, we will follow closely the strategy of \cite[Section 6]{Re2}, but replace counts of rational points over finite fields by motives as in the proof of \cite[Theorem 3.5]{RSW}. We denote by $K_0({\rm Var}/\mathbb{C})$ the free abelian group generated by representatives $[X]$ of isomorphism classes of complex varieties $X$, modulo the relation $[X]=[C]+[U]$ if $C$ is isomorphic to a closed subvariety of $X$ with open complement isomorphic to $U$. Multiplication in $K_0({\rm Var}/\mathbb{C})$ is given by $[X]\cdot[Y]=[X\times Y]$. We denote by $\mathbb{L}$ the class of the affine line; the following calculations will be performed in the localization $$\mathcal{K}=K_0({\rm Var}/\mathbb{C})[\mathbb{L}^{-1},(1-\mathbb{L}^n)^{-1}\, :\, n\geq 1].$$

At this point, we recall a notation from the subsection \ref{hns}. Let
\[
H:=\left\{ (s,d)^* \,:\, \exists\, r\in\cN_{>1},\,\text{$(s,d)^* = \big( (s^1,d^1),\ldots, (s^r,d^r) \big)$ is of HN-type} \right.
\]
\[
\left. \text{with weight $(s,d)$ and $s^r\neq0$}  \right\}. ~~~~~~~
\]
In this section, we write $\Rep_{(s,d)}$ for $\Rep_{(s,d)}\voll(A[T])$ and $\M_{(s,d)}^\sst  $ for $\M_{(s,d)}^\sst\big(A[T]\big)$. 

\subsection{Motive}
Obviously $H$ is finite and by Theorem \ref{StrukturMorphismusSurjektivCharakDurchHNFilt} we have
\[
\Rep_{(s,d)}~~=~~ \Rep_{(s,d)}^\sst ~\sqcup \bigsqcup_{(s,d)^*\in H} \Rep^{HN}_{(s,d)^*}
\]
and thus
\[
[ \Rep_{(s,d)}^\sst] ~~=~~ [ \Rep_{(s,d)}]
~-~ \sum_{(s,d)^*\in H} [ \Rep^{HN}_{(s,d)^*}]
\]
in $\mathcal{K}$. We then find
\[
\frac{[ \Rep_{(s,d)}^\sst]}{[\G_{(s,d)}]}  ~~=~~ \frac{[\Rep_{(s,d)}]}{[\G_{(s,d)}]}  
~-~ \sum_{(s,d)^*\in H} \frac{[ \Rep^{HN}_{(s,d)^*}]}{[\G_{(s,d)}]} .
\]

Fix $(s,d)^*\in H$. As in the proof of Theorem \ref{verAllgKin} we have
\begin{eqnarray*}
\Rep^{HN}_{(s,d)^*}  & \simeq & \G_{(s,d)}\times^{\mathcal{P}_{(s,d)^*}}\mathcal{Z}_0\,, \\
{[\mathcal{Z}_{0}]}& =& \mathbb{L}^{\sum_{n<l} \sum_{\alpha: i\to j} d_i^ld_j^n + s^l \langle t, d^n\rangle_Q} ~~\cdot~~\prod_{i=1}^r[\Rep_{(d^i,s^i)}^\sst],
\end{eqnarray*}
and we arrive at
\[
[\Rep^{HN}_{(s,d)^*}] ~~=~~
\frac{[ \G_{(s,d)}] }{[ \mathcal{P}_{(s,d)^*}]}
~\cdot~
\mathbb{L}^{\sum_{n<l} \sum_{\alpha: i\to j} d_i^ld_j^n + s^l \langle t, d^n\rangle_Q} ~~\cdot~~\prod_{i=1}^r[\Rep_{(d^i,s^i)}^\sst].
\]

This provides us with the motivic HN-recursion:
\begin{thm}
\[
\frac{[ \Rep_{(s,d)}^\sst]}{[\G_{(s,d)}]}  ~~~=~~~ \frac{[ \Rep_{(s,d)}]}{[\G_{(s,d)}]}  
~~~-~~ \sum_{(s,d)^*\in H}\frac{\mathbb{L}^{\sum_{n<l} \sum_{\alpha: i\to j} d_i^ld_j^n + s^l \langle t, d^n\rangle_Q}}{[ \Rep_{(s,d)^*}]}
~~\cdot~~ \prod_{i=1}^r [\Rep_{(d^i,s^i)}^\sst]\;\;\,.
\]
\end{thm}
Using the arguments of \cite[Theorem 6.7.]{Re2}, we see that the following relation 
\[
\sum_{i\in\cZ}\dim_{\cC} \mathbf{H}^i\big(\M_{(s,d)}^\sst,\mathbb{Q}\big)\mathbb{L}^{i/2}
~~=~~ [ \M_{(s,d)}^\sst  ]
\]
holds and we obtain: 
\begin{thm}\label{poincarepol} Let $(s,d)$ be a dimension vector such that semi-stability and stability  coincide. Then, with $\mathcal{P}\G_{(s,d)}:=\G_{(s,d)}/k^\times$ we have:
\[
\sum_{i\in\cZ}\dim_{\cC}\mathbf{H}^i\big(\M_{(s,d)}^\sst,\mathbb{Q}\big)\mathbb{L}^{i/2}~=~
\frac{[ \Rep_{(s,d)}]}{[\mathcal{P}\G_{(s,d)}]}  ~~~~-
\]
\[
(\mathbb{L}-1)\sum_{(s,d)^*\in H} \frac{1}{[ \mathcal{P}_{(s,d)^*}]} \cdot \mathbb{L}^{\sum_{n<l} \sum_{\alpha: i\to j} d_i^ld_j^n + s^l \langle t, d^n\rangle_Q} 
~~\cdot~~ \prod_{i=1}^r [\Rep_{(d^i,s^i)}^\sst ]\;\;\,.
\]
\end{thm}
\begin{proof}
As in \cite[Proposition 6.6.]{Re2} we have
\[
[ \M_{(s,d)}^\sst] ~= ~\frac{[ \Rep_{(s,d)}^\sst]}{[\mathcal{P}\G_{(s,d)}]},
\]
and the claim follows from the previous theorem.
\end{proof}

\subsection{Applications and examples}
\begin{lmm}Let $Q$ be of {Dynkin} type, let $\mathrm{Is}(Q,d)$ be the set theoretic quotient of $\Rep_{d}(Q)$ by the structure group $\G_{d}$ via the base change action, and for $M\in \Rep_{d}(Q)$ let
\[
\Hom_A^\mathrm{epi}(T,M):=\big\{ f\in\Hom_A(T,M) \,:\, \text{$f$ is surjective}\big\}.
\]
Then, we have
\[
[ \Rep_{(s,d)}]  ~~=~~
\sum_{ \substack{ [M]\in\mathrm{Is}(Q,d),\\ \exists\,T^s\epi M} } [ \mathcal{O}_M ]\, \cdot\, [\Hom_A^\mathrm{epi}(T^s,M) ].
\]
\end{lmm}
\begin{proof} We consider the map
\[
\Rep_{(s,d)}\to \bigg\{ M\in \Rep_d(Q)\,:\, \exists\, T^s\epi M \bigg\},\,(M,f)\mapsto M.
\]
and note that $[\Hom_A^\mathrm{epi}(T,?)]$ is constant along orbits.
\end{proof}

Overall, we find:
\begin{thm}\label{poincarepol_dynkin} Let $Q$ be of {Dynkin} type, and let $(s,d)$ be a dimension vector such that semi-stability and stability coincide. Then we have:
\[
\sum_{i\in\cZ}\dim_{\cC}\mathbf{H}^i\big(\M_{(s,d)}^\sst,\mathbb{Q}\big)\mathbb{L}^{i/2}~=~
\frac{1}{[\mathcal{P}\G_{(s,d)}]} ~~\cdot~~
\sum_{ \substack{ [M]\in\mathrm{Is}(Q,d),\\ \exists\,T^s\epi M} } [ \mathcal{O}_M] \,\cdot\, [\Hom_A^\mathrm{epi}(T^s,M) ]~~-~~
\]
\[
(\mathbb{L}-1)~~\cdot~~ \sum_{(s,d)^*\in H} \underbrace{\frac{1}{[ \mathcal{P}_{(s,d)^*}]} \cdot \mathbb{L}^{\sum_{n<l} \sum_{\alpha: i\to j} d_i^ld_j^n + s^l \langle t, d^n\rangle_Q} 
\cdot~~ \prod_{i=1}^r [\Rep_{(d^i,s^i)}^\sst]}_{=:\mathcal{S}_{(s,d)^*}}\;\;\,.
\]
\end{thm}
\begin{bsp}$ $\\Let 
\[
Q=\big( 1\to 2\big),~~T=\bigg( k^3\xto{[1,0,0]}k\bigg),
\]
and $(s,d)=(2,4,1)$. Then $H$ consists of
\[
(1,2,0)\, >_\mu(1,2,1),~~~~(1,1,1)\, >_\mu(1,3,0),~~~~(1,1,0)\, >_\mu(1,3,1).
\]
We calculate all components in the formula for each element in $H$:
\begin{eqnarray*}
\mathcal{S}_{(1,2,0)\, >_\mu(1,2,1)} & = & \frac{(\mathbb{L}^3-1)(\mathbb{L}^3-\mathbb{L})}{(\mathbb{L}-1)^4},\\
\mathcal{S}_{(1,1,1)\, >_\mu(1,3,0)} & = & \frac{1}{(\mathbb{L}-1)^2},\\
\mathcal{S}_{(1,1,0)\, >_\mu(1,3,1)} & = & \frac{(\mathbb{L}^3-1)}{\mathbb{L}(\mathbb{L}-1)^3}.\\
\end{eqnarray*}
We continue to calculate and get
\[
\frac{[ \Rep_{(s,d)}]}{[\mathcal{P}\G_{(s,d)}]}~=~
\frac{\mathbb{L}^2(\mathbb{L}^4-1)}{(\mathbb{L}-1)^2}~+~\frac{(\mathbb{L}^3-1)(\mathbb{L}^4-1)}{(\mathbb{L}-1)(\mathbb{L}^2-1)(\mathbb{L}^2-\mathbb{L})}.
\]
In total, we end up with:
\begin{eqnarray*}
\sum_{i\in\cZ}\dim_{\cC}\mathbf{H}^i\big(\M_{(s,d)}^\sst,\mathbb{Q}\big)\mathbb{L}^{i/2}
& = &
\frac{\mathbb{L}^2(\mathbb{L}^4-1)}{(\mathbb{L}-1)^2}~+~\frac{(\mathbb{L}^3-1)(\mathbb{L}^4-1)}{(\mathbb{L}-1)(\mathbb{L}^2-1)(\mathbb{L}^2-\mathbb{L})} \\
&  &~-~
\frac{(\mathbb{L}^3-1)(\mathbb{L}^3-\mathbb{L})}{(\mathbb{L}-1)^3}~-~\frac{1}{(\mathbb{L}-1)}~-~\frac{(\mathbb{L}^3-1)}{\mathbb{L}(\mathbb{L}-1)^2} \\
& = & 1+\mathbb{L}+\mathbb{L}^2+\mathbb{L}^3+\mathbb{L}^4.
\end{eqnarray*}
This result was to be expected if one remembers our explicit calculations of the moduli space.
\end{bsp}

\bibliographystyle{plain} 
\bibliography{main}

\end{document}